\documentclass[a4paper,10pt]{article}

\pdfoutput=1

\usepackage{amssymb}
\usepackage{graphicx}
\usepackage{amsthm}
\usepackage{bbm}
\usepackage{mathrsfs}
\usepackage{amsmath}
\usepackage[center]{caption}

\theoremstyle{plain}
\newtheorem{thm}{Theorem}[]
\newtheorem{cor}[thm]{Corollary}
\newtheorem{lem}[thm]{Lemma}
\newtheorem{prop}[thm]{Proposition}

\theoremstyle{definition}
\newtheorem{ex}{Example}
\newtheorem{defn}[thm]{Definition}
\newtheorem*{rmk}{Remark}

\newtheorem*{mainprf}{Proof of Theorem \ref{mainthm}}

\newcommand{\Qb}{\mathbb{Q}}
\newcommand{\Pb}{\mathbb{P}}
\newcommand{\Rb}{\mathbb{R}}
\newcommand{\Pt}{\tilde{\mathbb{P}}}
\newcommand{\Qt}{\tilde{\mathbb{Q}}}

\newcommand{\Fg}{\mathcal{F}}
\newcommand{\Gg}{\mathcal{G}}
\newcommand{\Ft}{\tilde{\mathcal{F}}}

\newcommand{\hs}{\hspace{2mm}}
\newcommand{\hsl}{\hspace{1mm}}
\newcommand{\ind}{\mathbbm{1}}

\author{Simon C. Harris and Matthew I. Roberts}
\title{Branching Brownian Motion:\\ Almost Sure Growth Along Unscaled Paths\\}

\begin{document}

\maketitle

\subsection*{Abstract}
We give new results on the growth of the number of particles in a dyadic branching Brownian motion which follow within a fixed distance of a path $f:[0,\infty)\to\Rb$. We show that it is possible to count the number of particles without rescaling the paths. Our results reveal that the number of particles along certain paths can oscillate dramatically. The methods used are entirely probabilistic, taking advantage of the spine technique developed by, amongst others, Lyons et al \cite{lyons_et_al:conceptual_llogl_mean_behaviour_bps}, Kyprianou \cite{kyprianou:alternative_simon_prob_analysis_fkpp}, and Hardy \& Harris \cite{hardy_harris:new_spine_formulation}.

\section{Introduction}

\noindent
The large-deviation properties of branching Brownian motion (BBM) have been well studied: for example, see Lee \cite{lee:large_deviations_for_branching_diffusions} and Hardy \& Harris \cite{hardy_harris:spine_bbm_large_deviations} for results on ``difficult'' paths which have a small probability of any particle following them, and Git \cite{git:almost_sure_path_properties} for the almost-sure growth rate of the number of particles along ``easy'' paths along which we see exponential growth in the number of particles. To give these results, the paths of a BBM are rescaled onto the interval $[0,1]$, echoing the approach of Schilder's theorem for a single Brownian motion.

Here we consider a problem similar in theme to the more classical path large deviations results of Git \cite{git:almost_sure_path_properties}, but from a naive standpoint, in which we are given a fixed function $f:[0,\infty)\to\Rb$ and we want to know how many particles in a BBM follow uniformly close to this path -- that is, within a fixed distance $L$ of $f(t)$ for all times $t\geq0$. Clearly there is a positive probability that no particle will achieve this (indeed, the very first particle could wander away from $f$ before it has the chance to give birth to another): in this event we say that the process becomes extinct.

The intuition is that the growth of the population due to branching is in constant competition with the ``deaths'' due to particles failing to follow the function $f$. Thus a natural condition arises: if the gradient of $f$ is too large, then the process eventually dies out almost surely; otherwise we may condition on non-extinction and give an almost sure result on the number of particles along the path.

We take advantage of the now well-known spine technique to interpret the change of measure given by a carefully chosen martingale. The change involves forcing one particle (the spine) to stay within a tube of radius $L$ about our function $f$ for all time. We then use the spine decomposition (see \cite{hardy_harris:new_spine_formulation}) which allows us to bound the growth of the system by looking at the births along the spine. We use only this intuitive tool, along with integration by parts, to complete the majority of the study -- and emphasise that the results follow so smoothly only because the appropriate choice of martingale allows the established spine methods to do the work for us.

We mentioned earlier that our results are given conditional on non-extinction. In fact, our proofs initially give results on the event that our particular martingale has a strictly positive limit. In Section \ref{inv_spine} we turn to showing that these events coincide to within a set of zero probability. The difficulties we face in this section are inherent in the time-inhomogeneity of the problem, and standard methods (analytic or probabilistic) cannot be applied. This fact is underlined by the observation that we are essentially considering a one-dimensional branching diffusion with \emph{time-dependent drift}, and asking how many particles remain within a bounded domain.

Finally, we note that our methods could easily be extended to a wide range of other branching diffusions. For simplicity, we consider only dyadic branching Brownian motion, but other diffusions and other branching distributions (subject to standard supercriticality and ``$A\log A$'' conditions) could be considered -- the spine techniques involved extend exactly as in the papers of Lyons et al \cite{kurtz_et_al:conceptual_kesten_stigum, lyons:simple_path_to_biggins, lyons_et_al:conceptual_llogl_mean_behaviour_bps}.

\section{Main result}

\subsection{Initial definitions}
We consider a branching Brownian motion starting with one particle at the origin, whereby each particle moves independently and undergoes independent dyadic branching at exponential rate $r>0$. We let the set of particles alive at time $t$ be $N(t)$, and for each particle $u \in N(t)$ denote its position at time $t$ by $X_u(t)$. This setup will be formalised later.

Fix a continuous function $f:[0,\infty)\rightarrow\Rb$. We say that $f$ satisfies the \emph{usual conditions} if:
\renewcommand{\labelenumi}{(\arabic{enumi})}
\begin{enumerate}
\item $f(0)=0$;
\item $f$ is twice continuously differentiable;
\item $\lim_{t\to\infty}\frac{1}{t}\int_0^t |f''(s)| ds = 0$.
\end{enumerate}
We assume unless otherwise stated that these conditions hold. After we obtain our results it will be possible to relax them slightly using simple uniform approximation arguments -- see Section \ref{extensions} -- but for now the stronger conditions on $f$ will allow us to apply integration by parts theorems without any complications.

Fix $L>0$ and let
\[S=S(f):=\limsup_{t\to\infty}\frac{1}{t}\int_0^t f'(s)^2 ds\]
and
\[\tilde S = r - \frac{\pi^2}{8L^2} - \frac{S}{2}.\]
Define
\[\hat N(t) = \left\{u \in N(t) : |X_u(s)-f(s)|< L \hs \forall s\leq t\right\},\]
the set of particles that have stayed within distance $L$ of the function $f$ for all times $s\leq t$. We wish to study the number of particles in $\hat{N}(t)$ at large times. Let
\[\Upsilon = \inf\{t\geq0 : \hat N(t)=\emptyset\}.\]
We call $\Upsilon$ the \emph{extinction time} for the process, and say that the process has become \emph{extinct} by time $t$ if $\Upsilon\leq t$. When we talk about \emph{non-extinction}, we mean the event $\Upsilon = \infty$.

\subsection{The main result}
We now state our main result. Most of this article will be concerned with proving this theorem.

\begin{thm}\label{mainthm}
If $\tilde S < 0$ or $\tilde S = - \infty$, then $\Upsilon<\infty$ almost surely.
On the other hand, if $\tilde S > 0$, then $\Pb(\Upsilon=\infty)>0$ and almost surely on non-extinction we have
\[\limsup_{t\to\infty}\frac{1}{t}\log|\hat{N}(t)| = r - \frac{\pi^2}{8L^2} - \liminf_{t\to\infty}\frac{1}{2t}\int_0^t f'(s)^2 ds\]
and
\[\liminf_{t\to\infty}\frac{1}{t}\log|\hat{N}(t)| = r - \frac{\pi^2}{8L^2} - \limsup_{t\to\infty}\frac{1}{2t}\int_0^t f'(s)^2 ds.\]
\end{thm}

\noindent
This theorem can be extended slightly to cover more general functions, and we give some results in this direction in Section \ref{extensions}.

\section{Examples}

\begin{ex}
Take $f(t) = \lambda t$. If $r < \frac{\lambda^2}{2} + \frac{\pi^2}{8L^2}$ then we have extinction almost surely; if $r > \frac{\lambda^2}{2} + \frac{\pi^2}{8L^2}$ then on non-extinction
\[\lim_{t\to\infty}\frac{1}{t}\log|\hat{N}(t)| = r - \frac{\pi^2}{8L^2} - \frac{\lambda^2}{2}.\]
For comparison, Git \cite{git:almost_sure_path_properties} gives a large deviations growth rate of $r - \lambda^2/2$ along such paths with $\lambda^2 < 2r$, so we see an extra cost of $\pi^2/8L^2$ for insisting that particles stay within a fixed distance $L$ of $f$ over the whole lifetime.
\end{ex}

\begin{ex}
Let $f(t) = t^\beta$, $\beta\in(0,1)$, or $f(t) = \log(t+1)$.
Provided that $r > \frac{\pi^2}{8L^2}$, on non-extinction we have
\[\lim_{t\to\infty}\frac{1}{t}\log|\hat{N}(t)| = r - \frac{\pi^2}{8L^2}.\]
Thus just as many particles follow these paths as the constant zero path. The same applies to any function with $S=0$ (provided that it satisfies the usual conditions). [Note that when trying to apply our result to $f(t)=t^\beta$, we have a small problem in that $f'(0)=\infty$. We can however approximate $f$ uniformly with, for example, $f_\varepsilon(t)= (t+\varepsilon)^\beta - \varepsilon^\beta$. For each $\varepsilon$ we get $S(f_\varepsilon) = 0$, and a very simple limiting argument gives the desired result.]
\end{ex}

\begin{ex}
Let $f(t) = \sqrt{2r}t - ct^\beta$, $\beta\in(0,1)$, or $f(t) = \sqrt{2r}t - c\log(t+1)$.
Then $S(f)=2r$ so we have extinction almost surely for any $L$ -- and the same applies to $f(t) = \sqrt{2t}- g(t)$ for any $g$ with $S(g)=0$.  This can be interpreted as saying that no particles travel for all time along any path ``close'' to criticality, and should be compared with the results of Bramson \cite{bramson:maximal_displacement_BBM} on the speed of the right-most particle.
\end{ex}

\begin{ex}
Let $f(t) = \lambda(t+1)\sin(\log(t+1))$. If $r < \frac{\lambda^2}{\sqrt5}\left(\frac{1+\sqrt5}{2}\right) + \frac{\pi^2}{8L^2}$ then we have extinction almost surely; if $r > \frac{\lambda^2}{\sqrt5}\left(\frac{1+\sqrt5}{2}\right) + \frac{\pi^2}{8L^2}$ then, on non-extinction, the number of particles alive at time $t$ oscillates, with
\[\liminf_{t\to\infty}\frac{1}{t}\log|\hat{N}(t)| = r - \frac{\pi^2}{8L^2} - \frac{\lambda^2}{\sqrt5}\left(\frac{\sqrt5+1}{2}\right)\]
and
\[\limsup_{t\to\infty}\frac{1}{t}\log|\hat{N}(t)| = r - \frac{\pi^2}{8L^2} - \frac{\lambda^2}{\sqrt5}\left(\frac{\sqrt5-1}{2}\right).\]
(Note the appearance of the golden ratio!)
\end{ex}

The reason for this oscillation becomes clearer when we consider the following simpler (but perhaps less natural) example.

\begin{ex}\label{nondiff}
Define a continuous function $f:[0,\infty)\to\Rb$ by setting $f(t)=0$ for $t\in[0,1]$ and
\[f'(t) = \left\{ \begin{array}{ll} 0 & \hbox{ if } \hsl 2^{2k} \leq t < 2^{2k+1} \hsl \hbox{ for some } \hsl k\in\{0,1,2,\ldots\}\\
									1 & \hbox{ if } \hsl 2^{2k+1} \leq t < 2^{2k+2} \hsl \hbox{ for some } \hsl k\in\{0,1,2,\ldots\}\end{array}\right. .\]
Then, provided that $r > \frac{1}{3} + \frac{\pi^2}{8L^2}$, on non-extinction we have
\[\liminf_{t\to\infty}\frac{1}{t}\log|\hat{N}(t)| = r - \frac{\pi^2}{8L^2} - \frac{1}{3}\]
and
\[\limsup_{t\to\infty}\frac{1}{t}\log|\hat{N}(t)| = r - \frac{\pi^2}{8L^2} - \frac{1}{6}.\]
The idea here is that the number of particles grows quickly when $f'(t)=0$, but much more slowly when $f'(t)=1$ as the steep gradient means that particles have to struggle to follow the path for a long time. As the size of the intervals $[2^n, 2^{n+1}]$ grows exponentially, the behaviour of the number of particles at time $t$ is dominated by the behaviour on the most recent such interval. [We note that this choice of $f$ is not twice differentiable; however, it can be uniformly approximated by twice differentiable functions, and it is easily checked that our results still hold.]
\end{ex}

\section{The spine setup}

\noindent
Consider a dyadic one-dimensional branching Brownian motion, branching at rate $r$, with associated probability measures $\Pb_x$ under which
\begin{itemize}
\item{we begin with a root particle, $\emptyset$, at $x$;}
\item{if a particle $u$ is in the tree then all its ancestors, denoted $\{v: v<u\}$, are also in the tree;}
\item{each particle $u$ has a lifetime $\sigma_u$, which is exponentially distributed with parameter $r$, and a fission time $S_u = \sum_{v\leq u}\sigma_v$;}
\item{at the fission time $S_u$, $u$ has disappeared and been replaced by two children $u0$ and $u1$, which inherit the position of their parent;}
\item{each particle $u$ has a position $X_u(t) \in \Rb$ at each time $t\in [S_u-\sigma_u, S_u)$;}
\item{each particle $u$, while alive, moves according to a standard Brownian motion started from $X_u(S_u-\sigma_u)$.}
\end{itemize}
For convenience, we extend the position of a particle $u$ to all times $t\in[0, S_u)$, to include the paths of all its ancestors:
\[X_u(t):= X_v(t) \hbox{ if } v\leq u \hbox{ and } S_v - \sigma_v  \leq t < S_v.\]
We recall that we defined $N(t)$ to be the set of particles alive at time $t$,
\[N(t):=\{u: S_u - \sigma_u \leq t < S_u\},\]
and also that
\[\hat N(t) := \left\{u \in N(t) : |X_u(s)-f(s)|< L \hs \forall s\leq t\right\}.\]

We choose from our BBM one distinguished line of descent or \emph{spine} -- that is, a subset $\xi$ of the tree such that $\xi\cap N(t)$ contains exactly one particle for each $t$ and if $u\in \xi$ and $v<u$ then $v\in \xi$. We make this choice as follows:
\begin{itemize}
\item{the initial particle $\emptyset$ is in the spine;}
\item{at the fission time of node $u$ in the spine, the new spine particle is chosen uniformly at random from the two children $u0$ and $u1$ of $u$.}
\end{itemize}
We call the resulting probability measure (on the space of \emph{marked trees with spines}) $\Pt_x$. The full construction of $\Pt_x$ can be found in \cite{hardy_harris:new_spine_formulation}.

\subsection{Filtrations}
We use three different filtrations, $\Fg_t$, $\Ft_t$ and $\Gg_t$, to encapsulate different amounts of information. We give descriptions of these filtrations here, but the reader is referred to \cite{hardy_harris:new_spine_formulation} for the full definitions.

\begin{itemize}
\item{$\Fg_t$ contains the all the information about the marked tree up to time $t$. However, it does not know which particle is the spine at any point.}
\item{$\Ft_t$ contains all the information about both the marked tree and the spine up to time $t$.}
\item{$\Gg_t$ contains just the spatial information about the spine up to time $t$; it does not know anything about the rest of the tree.}
\end{itemize}
We note that $\Fg_t \subseteq \Ft_t$ and $\Gg_t \subseteq \Ft_t$, and also that $\Pt_x$ is an extension of $\Pb_x$ in that $\Pb_x = \Pt_x |_{\Fg_\infty}$.

\subsection{Martingales and a change of measure}\label{measure_change}

Under $\Pt$, the path of the spine $(\xi_t,\hsl t\geq 0)$ is simply a Brownian motion, and thus we can apply It\^{o}'s formula to see that
\[V_t := e^{\pi^2 t / 8L^2} \cos\left(\frac{\pi}{2L}(\xi_t-f(t))\right)e^{\int_0^t f'(s) d\xi_s - \frac{1}{2}\int_0^t f'(s)^2 ds}\]
is a $\Gg_t$-martingale. By stopping the process at the first exit time of the spine particle from the tube $\{(x,t): |f(t)-x|< L\}$, we obtain also that
\[\zeta(t) := e^{\pi^2 t / 8L^2} \cos\left(\frac{\pi}{2L}(\xi_t-f(t))\right)e^{\int_0^t f'(s) d\xi_s - \frac{1}{2}\int_0^t f'(s)^2 ds} \ind_{\{|f(s)-\xi_s| < L \hsl \forall s\leq t\}}\]
is a $\Gg_t$-martingale. We call this martingale $\zeta$ the \emph{single-particle martingale}.

\begin{defn}
We define an $\Ft_t$-adapted martingale by
\[\tilde{\zeta}(t) = 2^{|\xi_t|} \times e^{-rt} \times \zeta(t),\]
where $|\xi_t|$ denotes the generation of the spine at time $t$, $|\xi_t|=|\{v: v<\xi_t\}|$. The proof that this process is an $\Ft_t$-martingale is given in \cite{hardy_harris:new_spine_formulation}.

We note that if $f$ is an $\Ft_t$-measurable function then we can write:
\begin{equation}
f(t)=\sum_{u\in N_t}f_u(t) \ind_{\xi_t=u} \label{fdecomp}
\end{equation}
where each $f_u$ is $\Fg_t$-measurable. It is also shown in \cite{hardy_harris:new_spine_formulation} that if we define
\[Z(t) := \sum_{u\in N(t)} e^{-rt}\zeta_u(t),\]
where $\zeta_u$ is the $\Fg_t$-adapted process defined via the representation of $\zeta$ as in (\ref{fdecomp}), then
\[Z(t) = \Pt[\tilde \zeta(t) | \Fg_t].\]
One may easily use this representation to show that $Z$ is an $\Fg_t$-martingale. This martingale is the main object of interest, and we write it out in full:
\[Z(t) = \sum_{u\in \hat N(t)} e^{(\pi^2 / 8L^2 - r)t} \cos\left(\frac{\pi}{2L}(X_u(t)-f(t))\right)e^{\int_0^t f'(s) dX_u(s) - \frac{1}{2}\int_0^t f'(s)^2 ds}.\]
\end{defn}

\begin{defn}
We define a new measure, $\Qt_x$, via
\[\left.\frac{d\Qt_x}{d\Pt_x}\right|_{\Ft_t} = \frac{\tilde{\zeta}(t)}{\tilde{\zeta}(0)}.\]
Also, for convenience, define $\Qb_x$ to be the projection of the measure $\Qt$ onto $\Fg_\infty$; then
\[\left.\frac{d\Qb_x}{d\Pb_x}\right|_{\Fg_t} = \frac{Z(t)}{Z(0)}.\]
\end{defn}

\begin{lem}
Under $\Qt_x$,
\begin{itemize}
\item when at position $x$ at time $t$ the spine $\xi$ moves as a Brownian motion with drift
\[f'(t) - \frac{\pi}{2L} \tan\left(\frac{\pi}{2L}(x - f(t))\right);\]
\item the fission times along the spine occur at an accelerated rate $2r$;
\item at the fission time of node $v$ on the spine, the single spine particle is replaced by two children, and the new spine particle is chosen uniformly from the two children;
\item the remaining child gives rise to an independent subtree, which is not part of the spine and is determined by an independent copy of the original measure $\Pb$ shifted to the position and time of creation.
\end{itemize}
\end{lem}

\noindent
Thus, under $\Qt_x$, the spine remains within distance $L$ of $f(t)$ for all times $t\geq0$. To see this explicitly, note that
\[\Qt_x(\xi_t\not\in\hat{N}(t)) = \Pt_x\left[\ind_{\{\xi_t\not\in\hat{N}(t)\}}\tilde\zeta(t)\right] = 0\]
by definition of $\tilde\zeta(t)$. All other particles, once born, move like independent standard Brownian motions but -- as under $\Pb_x$ -- we imagine them being ``killed'' instantly upon leaving the tube of width $2L$ about $f$. In reality they are still present in the system, but make no contribution to $Z$ once they have left the tube.

It is possible to show that the motion of the process $\xi_t - f(t)$ has equilibrium distribution
\[\mu(dx)=\frac{1}{L}\cos^2\left(\frac{\pi x}{2L}\right)\ind_{\{x\in(-L,L)\}}dx,\]
although we will not need to use this property.

\begin{rmk}
Note that $\hat{N}$, and hence $Z$ and $\Qt$, depend upon the function $f$ and the constant $L$. Usually these will be implicit, but occasionally we shall write $\hat{N}^{f,L}$, $Z^{f,L}$ and $\Qt^{f,L}$ to emphasise the choice of $f$ and $L$ in use at the time.
\end{rmk}

\subsection{Spine tools}

We now state the spine decomposition theorem, which will be a vital tool in our investigation. It allows us to relate the growth of the whole process to just the behaviour along the spine. For a proof the reader is again referred to \cite{hardy_harris:new_spine_formulation}.

\begin{thm}[Spine decomposition]
We have the following decomposition of $Z$:
\[\Qt_x[Z(t)|\Gg_\infty] = \int_0^t 2r e^{-rs}\zeta(s) ds + e^{-rt}\zeta(t).\]
\end{thm}

The spine decomposition is usually used in conjunction with a result like the following -- a proof of a more general form of this lemma can be found in \cite{lyons_peres:probability_on_trees}.

\begin{lem}
\label{stdmeas}
Let $Z(\infty)=\limsup Z(t)$. Then
\[\Qb \ll \Pb \hs \Leftrightarrow \hs Z(\infty)<\infty \hs \Qb\hbox{-a.s. } \Leftrightarrow \hs \Qb = Z(\infty) \Pb\]
and
\[\Qb \perp \Pb \hs \Leftrightarrow \hs Z(\infty)=\infty \hs \Qb\hbox{-a.s. } \Leftrightarrow \hs \Pb[Z(\infty)]=0.\]
\end{lem}

Another extremely useful spine tool -- also proved in \cite{hardy_harris:new_spine_formulation} -- is the \emph{many-to-one} theorem. A much more general version of this theorem is given in \cite{hardy_harris:new_spine_formulation}, but the following version will be enough for our purposes.

\begin{thm}[Many-to-One]
\label{many_to_one}
If $f(t)$ is $\Gg_t$-measurable for each $t\geq0$ with representation \emph{(\ref{fdecomp})}, then
\[\Pb[\sum_{u\in N(t)} f_u(t)] = e^{rt}\Pt[f(t)].\]
\end{thm}

We have one more lemma, a proof of which can be found in \cite{harris_roberts:extinction_letter}. Although this result is extremely simple -- and essential to our study -- we are not aware of its presence in the literature before \cite{harris_roberts:extinction_letter}.

\begin{lem}
\label{extinction_lem}
For any $t\in[0,\infty]$ (note that infinity is included here), we have
\[\Pb_x(Z(t)>0) = \Qb_x\left[\frac{Z(0)}{Z(t)}\right].\]
\end{lem}

\section{Almost sure growth along paths}

\subsection{Controlling the measure change}
Before applying the tools that we have developed, we need the following short lemma to keep the Girsanov part of our change of measure under control.

\begin{lem}
\label{int_lem}
For any $u\in\hat{N}(t)$, almost surely under both $\Pt_x$ and $\Qt_x$ we have
\[\left|\int_0^t f'(s) dX_u(s) - \int_0^t f'(s)^2 ds\right| \leq 2L\int_0^t |f''(s)| ds + 2L|f'(0)|.\]
\end{lem}

\begin{proof}
From the integration by parts formula for It\^o calculus, we know that
\[f'(t)X_u(t) = f'(0)X_u(0) + \int_0^t f''(s)X_u(s) ds + \int_0^t f'(s)dX_u(s).\]
From ordinary integration by parts,
\[\int_0^t f'(s)^2 ds = f'(t)f(t) - f'(0)f(0) - \int_0^t f(s)f''(s) ds.\]
We also note that, if $u\in\hat N(t)$ then $|X_u(s) - f(s)|< L$ for all $s\leq t$. Thus
\begin{eqnarray*}
&&\left|\int_0^t f'(s) dX_u(s) - \int_0^t f'(s)^2 ds\right|\\
&&\leq |f'(t)(X_u(t)-f(t)) - f'(0)(X_u(0)-f(0)) - \int_0^t f''(s)(X_u(s)-f(s))ds|\\
&&\leq 2L\int_0^t |f''(s)| ds + 2L|f'(0)|.\qedhere
\end{eqnarray*}
\end{proof}

The above estimate motivates the following definition:

\begin{defn}
For $p\in[0,1)$ set
\begin{equation*}
T(p) = \inf\{t : \textstyle{\int_0^s (r - \frac{\pi^2}{8L^2} - \frac{1}{2}f'(u)^2 - 2L|f''(u)|)du - 2L|f'(0)|} \geq p\tilde S s \hs \forall s\geq t\}.
\end{equation*}
We note that $T(p)$ is deterministic and finite.
\end{defn}

We are now ready to give our first real result, which tells us when our measure change is well-behaved.

\begin{prop}\label{uiprop}
Recall that $Z(\infty):= \limsup_{t\to\infty}Z(t)$. If $\tilde S < 0$ or $\tilde S = -\infty$, then the process almost surely becomes extinct in finite time (and hence we have $Z(\infty)=0$). Alternatively, if $\tilde S > 0$ then $\Pb[Z(\infty)]=1$.
\end{prop}

\begin{proof}
Suppose first that $\tilde S \in [-\infty,0)$. Then $r < \frac{S}{2} + \frac{\pi^2}{8L^2}$ so we may choose $L' > L$ and finite $S'\leq S$ such that
\[r < \frac{S'}{2} + \frac{\pi^2}{8L'^2}.\]
Let $\eta = \cos(\pi L/2L')$ and $\tilde S' = r - \pi^2/8L'^2 - S'/2$.
Since $L'>L$, we have
\[\hat{N}^{f,L}(t)\neq\emptyset \hsl \Rightarrow \hsl Z^{f,L'}(t)>0.\]
Recall the extinction time $\Upsilon:= \inf\{t\geq0: \hat{N}(t)=\emptyset\}$. Then
\begin{eqnarray*}
\Pb(\Upsilon = \infty) &=& \lim_{t\to\infty}\Pb(\hat{N}^{f,L}(t)\neq\emptyset)\\
&=& \lim_{t\to\infty}\Pb\left[\frac{Z^{f,L'}(t)}{Z^{f,L'}(t)}\ind_{\{\hat{N}^{f,L}(t)\neq\emptyset\}}\right]\\
&=& \lim_{t\to\infty}\Qb^{f,L'}\left[\frac{1}{Z^{f,L'}(t)}\ind_{\{\hat{N}^{f,L}(t)\neq\emptyset\}}\right]\\
&\leq& \lim_{t\to\infty}\Qb^{f,L'}\left[\frac{\ind_{\{\hat{N}^{f,L}(t)\neq\emptyset\}}}{\sum_{u\in\hat{N}^{f,L}(t)} \eta e^{(\frac{\pi^2}{8L'^2} - r)t + \int_0^t f'(s)dX_u(s) - \frac{1}{2}\int_0^t f'(s)^2 ds} }\right].
\end{eqnarray*}
If $\hat{N}^{f,L}(t)\neq\emptyset$ then there is at least one particle in $\hat{N}^{f,L}(t)$: we may apply Lemma \ref{int_lem} to its term in the denominator above to get
\begin{eqnarray*}
\Pb(\Upsilon = \infty) &\leq& \lim_{t\to\infty}\frac{1}{\eta}\Qb^{f,L'}\left[\frac{1}{e^{(\frac{\pi^2}{8L'^2}-r)t + \frac{1}{2}\int_0^t f'(s)^2 ds - 2L\int_0^t |f''(s)|ds - 2L|f'(0)|}}\right]\\
&\leq& \lim_{t\to\infty}\frac{1}{\eta}\frac{1}{ e^{-\tilde S' t + o(t)} } = 0,\\
\end{eqnarray*}
which proves our first claim.

\vspace{2mm}

Now suppose that $\tilde S > 0$. We recall the spine decomposition:
\[\Qt[Z(t)|\Gg_\infty] = \int_0^t 2r e^{-rs}\zeta(s) ds + e^{-rt}\zeta(t).\]
Since, under $\Qt$, the spine is almost surely in $\hat N(t)$ for each $t\geq 0$, we may use Lemma \ref{int_lem} to bound both terms: for any $p\in(0,1)$ and $t\geq T(p)$,
\begin{eqnarray*}
e^{-rt}\zeta(t) &=& e^{(\frac{\pi^2}{8L^2} - r)t + \int_0^t f'(s) d\xi_s - \frac{1}{2}\int_0^t f'(s)^2 ds}\cos\left(\frac{\pi}{2L}(\xi_t-f(t))\right)\\
&\leq& e^{-\int_0^t (r - \frac{\pi^2}{8L^2} - \frac{1}{2} f'(s)^2 ds - 2L|f''(s)|) ds + |f'(0)|} \hs \leq \hs e^{-p\tilde S t}
\end{eqnarray*}
so that
\[\Qt[Z(t)|\Gg_\infty] \leq \int_0^{T(p)}2r e^{-rs}\zeta(s) ds + \int_{T(p)}^t 2r e^{-p\tilde S s} ds + e^{-p\tilde S t},\]
and thus $\liminf_{t\to\infty} \Qt[Z(t)|\Gg_\infty] < \infty$ $\Qt$-almost surely. It is easily checked that $1/Z$ is a positive supermartingale under $\Qt$, and hence $Z(t)$ converges $\Qt$-almost surely to some (possibly infinite) limit. Thus, applying Fatou's lemma, we get
\[\Qt[Z(\infty)|\Gg_\infty]\leq \liminf_{t\to\infty} \Qt[Z(t)|\Gg_\infty] <\infty.\]
We deduce that $Z(\infty)<\infty$ $\Qt$-almost surely, and Lemma \ref{stdmeas} then gives that $\Pb[Z(\infty)]=1$.
\end{proof}

\subsection{Almost sure growth}

The two propositions in this section contain the meat of our results. Proposition \ref{liminf_prop} gives a lower bound on the number of particles in $\hat{N}(t)$ for large $t$, and Proposition \ref{limsup_prop} an upper bound. The former holds only on the event that $Z$ has a positive limit; as mentioned in the introduction, this set coincides (up to a null event) with the event that no particle manages to follow within $L$ of $f$, although we will not prove this fact until later. The proofs of our two propositions are very simple, but we stress again that this is due to the careful choice of martingale.

\begin{prop}\label{liminf_prop}
Let $\Omega^{\star}$ be the set on which $Z$ has a strictly positive limit,
\[\Omega^{\star} := \left\{\liminf_{t\to\infty} Z(t)>0 \right\}.\]
Then almost surely on $\Omega^{\star}$ we have
\[\liminf_{t\to\infty} \frac{1}{t}\log |\hat N^{f,L}(t)| \geq r - \frac{\pi^2}{8L^2} - \limsup_{t\to\infty}\frac{1}{2t}\int_0^t f'(s)^2 ds\]
and
\[\limsup_{t\to\infty} \frac{1}{t}\log |\hat N^{f,L}(t)| \geq r - \frac{\pi^2}{8L^2} - \liminf_{t\to\infty}\frac{1}{2t}\int_0^t f'(s)^2 ds.\]
\end{prop}

\begin{proof}
For any $t\geq 0$, by Lemma \ref{int_lem}, almost surely under $\Pb$,
\begin{eqnarray*}
Z(t) &=& \sum_{u\in \hat N(t)} e^{(\pi^2 / 8L^2 - r)t} \cos\left(\frac{\pi}{2L}(X_u(t)-f(t))\right)e^{\int_0^t f'(s) dX_u(s) - \frac{1}{2}\int_0^t f'(s)^2 ds}\\
&\leq& |\hat N(t)| e^{(\pi^2/8L^2-r)t + \frac{1}{2}\int_0^t f'(s)^2 ds + 2L\int_0^t |f''(s)| ds + 2L|f'(0)|}.
\end{eqnarray*}
Hence
\[\frac{1}{t}\log |\hat N(t)| \geq \frac{1}{t}\log Z(t) + r - \frac{\pi^2}{8L^2} - \frac{1}{2t} \int_0^t f'(s)^2 ds - \frac{2L}{t}\int_0^t |f''(s)| ds - \frac{2L}{t}|f'(0)|.\]
Now, on $\Omega^{\star}$ we have $\liminf_{t\to\infty} Z(t) >0$ and thus
\[\liminf_{t\to\infty} \frac{1}{t}\log Z(t) \geq 0.\]
It is then a simple exercise, using that $|\hat N(t)|$ and $Z(t)$ are c\`adl\`ag functions of $t$, to show that
\[\liminf_{t\to\infty}\frac{1}{t}\log |\hat N(t)| \geq r - \frac{\pi^2}{8L^2} - \limsup_{t\to\infty}\frac{1}{2t}\int_0^t f'(s)^2 ds.\]
On the other hand, taking (deterministic) times $t_n\to\infty$ such that
\[\lim_{n\to\infty}\frac{1}{t_n}\int_0^{t_n} f'(s)^2 ds = \liminf_{t\to\infty}\frac{1}{t}\int_0^t f'(s)^2 ds\]
and running the same argument as above along the sequence $t_n$, we get
\begin{eqnarray*}
\limsup_{t\to\infty}\frac{1}{t}\log |\hat N(t)| &\geq& \liminf_{n\to\infty}\frac{1}{t_n}\log |\hat N(t_n)|\\
&\geq& r - \frac{\pi^2}{8L^2} - \limsup_{n\to\infty}\frac{1}{2t_n}\int_0^{t_n} f'(s)^2 ds\\
&=& r - \frac{\pi^2}{8L^2} - \liminf_{t\to\infty}\frac{1}{2t}\int_0^t f'(s)^2 ds. \qedhere
\end{eqnarray*}
\end{proof}

\begin{rmk}
Recall that under $\Pb$, $Z$ is a positive martingale so $\liminf_{t\to\infty}Z(t) = Z(\infty)$ $\Pb$-almost surely. If $\tilde S > 0$, then $\Pb[Z(\infty)]=1$, so in this case $\Omega^{\star}$ occurs with strictly positive probability.
\end{rmk}

\begin{prop}\label{limsup_prop}
For any $S\in[0,\infty]$ and $L>0$, $\Pb$-almost surely we have
\[\limsup_{t\to\infty} \frac{1}{t}\log |\hat N^{f,L}(t)| \leq r - \frac{\pi^2}{8L^2} - \liminf_{t\to\infty}\frac{1}{2t}\int_0^t f'(s)^2 ds\]
and
\[\liminf_{t\to\infty} \frac{1}{t}\log |\hat N^{f,L}(t)| \leq r - \frac{\pi^2}{8L^2} - \limsup_{t\to\infty}\frac{1}{2t}\int_0^t f'(s)^2 ds.\]
\end{prop}

\begin{proof}
Fix $\alpha>1$ and let $\varepsilon=\cos(\pi/2\alpha)$. Since $Z^{f,\alpha L}$ is a positive martingale under $\Pb$, we have $Z^{f,\alpha L}(\infty) < \infty$ $\Pb$-almost surely. This implies that, almost surely,
\[\limsup_{t\to\infty}\frac{1}{t}\log Z^{f,\alpha L}(t) \leq 0.\]
Now, almost surely under $\Pb$,
\[Z^{f,\alpha L}(t) \hs = \sum_{u\in \hat N^{f,\alpha L}(t)} e^{-rt} \zeta^{f,\alpha L}_u(t) \hs \geq \sum_{u\in \hat N^{f,L}(t)} e^{-rt} \zeta^{f,\alpha L}_u(t).\]
By the definition of $\varepsilon$ above, for any $u\in \hat N^{f,L}(t)$ the cosine term in $\zeta^{f,\alpha L}_u(t)$ is at least $\varepsilon$ (since the particle is within $L$ of $f(t)$ at time $t$). Applying Lemma \ref{int_lem} we see that
\[Z^{f,\alpha L}(t) \geq |\hat N^{f,L}| e^{(\frac{\pi^2}{8\alpha^2 L^2}-r)t}\cdot\varepsilon\cdot e^{\frac{1}{2}\int_0^t f'(s)^2 ds - 2L\int_0^t |f''(s)| ds - 2L|f'(0)|}\]
and hence
\begin{eqnarray*}
\frac{1}{t}\log |\hat N^{f,L}(t)| &\leq& \frac{1}{t}\log Z^{f,\alpha L}(t) + r - \frac{\pi^2}{8\alpha^2 L^2} +\frac{1}{t}\log\frac{1}{\varepsilon}\\
									&&   \hspace{8mm} + \frac{1}{2t}\int_0^t f'(s)^2 ds - \frac{2L}{t}\int_0^t |f''(s)|ds - \frac{2L}{t}|f'(0)|.
\end{eqnarray*}
Thus (using that $|\hat N^{f,L}(t)|$ and $Z^{f,\alpha L}(t)$ are c\`adl\`ag functions of $t$) we may easily show that
\[\limsup_{t\to\infty}\frac{1}{t}\log |\hat N^{f,L}(t)| \leq r - \frac{\pi^2}{8\alpha^2 L^2} - \liminf_{t\to\infty}\frac{1}{2t}\int_0^t f'(s)^2 ds.\]
Our first claim follows by letting $\alpha\downarrow 1$.
Now, taking times $s_n\to\infty$ such that
\[\lim_{n\to\infty}\frac{1}{s_n}\int_0^{s_n} f'(s)^2 ds = \limsup_{t\to\infty}\frac{1}{t}\int_0^t f'(s)^2 ds\]
and running the same argument as above along the sequence $s_n$, we get
\begin{eqnarray*}
\liminf_{t\to\infty}\frac{1}{t}\log |\hat N(t)| &\leq& \limsup_{n\to\infty}\frac{1}{s_n}\log |\hat N(s_n)|\\
&\leq& r - \frac{\pi^2}{8L^2} - \liminf_{n\to\infty}\frac{1}{2s_n}\int_0^{s_n} f'(s)^2 ds\\
&=& r - \frac{\pi^2}{8L^2} - \limsup_{t\to\infty}\frac{1}{2t}\int_0^t f'(s)^2 ds. \qedhere
\end{eqnarray*}
\end{proof}

\section{Showing that $Z(\infty)=0$ agrees with extinction}\label{inv_spine}
We note that we have now established our main result except for one key point: we have been working so far on the event $\{Z(\infty)>0\}$, rather than the event of non-extinction of the process, $\{\Upsilon = \infty\}$. We turn now to showing that these two events differ only on a set of zero probability.

The approach to proving this is often analytic, showing that \mbox{$\Pb(Z(\infty)>0)$} and \mbox{$\Pb($non-extinction$)$} satisfy the same differential equation with the same boundary conditions, and then showing that any such solution to the equation is unique. There is sometimes a probabilistic approach to such arguments: one considers the product martingale
\[P(t):=\Pb(Z(\infty)=0 | \Fg_t) = \prod_{u\in N(t)} \Pb_{X_u(t)}(Z_u(\infty)=0).\]
On extinction, the limit of this process is clearly 1, and if we could show that on non-extinction the limit is 0, then since $P$ is a bounded non-negative martingale we would have
\[\Pb(\hbox{extinction}) = \Pb[P(\infty)] = \Pb[P(0)] = \Pb(Z(\infty)=0).\]
In \cite{harris_et_al:fkpp_one_sided_travelling_waves}, for example, we have killing of particles at the origin rather than on the boundary of a tube -- and it is shown that on non-extinction, at least one particle escapes to infinity and its term in the product martingale tends to zero. This is enough to complete the argument (although in \cite{harris_et_al:fkpp_one_sided_travelling_waves} the authors favour the analytic approach). In our case we are hampered by the fact that for a single particle $u$ the value of $\Pb_{X_u(t)}(Z_u(\infty)=0)$ is bounded away from zero, and if the particle is close to the edge of the tube, or even possibly in some places in the interior the tube, then this probability takes values arbitrarily close to 1.

The time-inhomogeneity of our problem means that other standard methods also fail. Our alternative approach is more direct: we show that if at least one particle survives for a long time, then it will have many births in ``good'' areas of the tube, and thus $Z(\infty)>0$ with high probability.

Recall that under $\Pt_x$, we start at time $t=0$ with one particle at position $x$ (rather than at the origin) -- and similarly for $\Qt_x$. We now need some more notation.

\begin{defn}
For $t\in[0,\infty)$ define
\[\begin{array}{lrcl} g_t: & [0,\infty) &\to     &\Rb\\
						   &      s     &\mapsto &f(s+t)-f(t).\end{array}\]
Now for $\alpha\in[0,1)$, define
\[U_\alpha = \{(t,x) : \Pb_{x-f(t)}(Z^{g_t, L}(\infty)>0) \geq \alpha\} \subseteq [0,\infty)\times\Rb.\]
Finally, for any particle $u$ and $t\geq0$, define
\[I_\alpha(u;t) = \int_0^{t\wedge S_u} \ind_{\{X_u(s) \in U_\alpha\}} ds;\]
$I_\alpha(u;t)$ is the time spent by particle $u$ in the set $U_\alpha$ before $t$.
\end{defn}

\begin{figure}[h!]
  \centering
      \includegraphics[width=\textwidth]{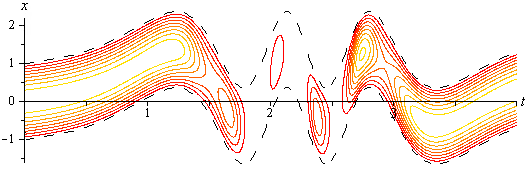}
  \caption{Approximation to a section of $U_\alpha$ for eight different values of $\alpha$ when \mbox{$f(t)=\sin(a \tanh(t+b))+c$} for some constants $a$, $b$ and $c$.}
\end{figure}

Our first lemma in this section establishes that for sufficiently small $\alpha$, $U_\alpha$ -- which we think of as the good part of the tube -- stretches to near the top and bottom edges of the tube for almost $\tilde S/r$ proportion of the time. To do this we use the identity given in Lemma \ref{extinction_lem} combined with the spine decomposition.

\begin{lem}\label{tube_filling}
Fix $\delta\in(0,L)$ and $\beta<1$. If $\tilde S > 0$ then for sufficiently small $\alpha>0$ and large $T$, we have
\[\int_0^t \ind_{\{(s,x)\in U_\alpha \hsl \forall x \in [-L+\delta, L-\delta]\}} ds \geq \beta\frac{\tilde S}{r}t \hs \forall t\geq T.\]
\end{lem}

\begin{proof}
Fix $q\in(0,\frac{1-\beta}{2})$ and $p\in(\beta+2q,1)$; we show that for $\alpha = \frac{q\tilde S \cos(\pi\delta / 2L)}{2re^{2Lr\sqrt{2/q\tilde S}}}$ and $t\geq T(p)$, we have
\[\int_0^t \ind_{\{(s,x)\in U_\alpha \hsl \forall x \in [-L+\delta, L-\delta]\}} ds \geq (p-2q)\frac{\tilde S}{r}t.\]
Let
\[J_t = \inf_{s\geq t} \int_0^s \left(r - \frac{\pi^2}{8L^2} - \frac{1}{2}f'(u)^2 - 2L|f''(u)| - q \tilde S\right) du,\]
and define two subsets, $U$ and $V$, of $[0,\infty)$ by
\[U = \{t\geq0: J_t \hbox{ is increasing at } t\} \hs \hbox{ and } \hs V = \left\{t\geq 0 : |f'(t)|< r\sqrt{2/q\tilde S}\right\}.\]
If $J$ is increasing at $t$, then clearly for any $s>0$
\begin{multline*}
\hspace{5mm}\int_0^{t+s}( r - \frac{\pi^2}{8L^2} - \frac{1}{2}f'(u)^2 - 2L|f''(u)| - q\tilde S) du\\
> \int_0^t ( r - \frac{\pi^2}{8L^2} - \frac{1}{2}f'(u)^2 - 2L|f''(u)| - q\tilde S) du,\hspace{5mm}
\end{multline*}
and hence
\[\int_t^{t+s} (r - \frac{\pi^2}{8L^2} - \frac{1}{2}f'(u)^2 - 2L|f''(u)|) du > q\tilde S s.\]
Thus if $t\in U\cap V$ then, as in Proposition \ref{uiprop}, we can apply the spine decomposition and Lemma \ref{int_lem} to get
\begin{eqnarray*}
\Qt_x[Z^{g_t,L}(\infty) | \Gg_\infty] &\leq& \int_0^\infty 2r e^{(\frac{\pi^2}{8L^2} - r)s + \int_0^s g_t'(u) d\xi_u - \frac{1}{2}\int_0^s g_t'(u)^2 du} ds\\
&\leq& \int_0^\infty 2r e^{(\frac{\pi^2}{8L^2} - r)s + \frac{1}{2}\int_0^s g_t'(u)^2 du + 2L\int_0^s |g_t''(u)|du + 2L|g_t'(0)|} ds\\
&=& \int_0^\infty 2r e^{-\int_t^{t+s} (r-\frac{\pi^2}{8L^2} - \frac{1}{2}f'(u)^2 - 2L|f''(u)|)du + 2L|f'(t)|} ds\\
&\leq& e^{2Lr\sqrt{2/q\tilde S}}\int_0^\infty 2r e^{-q\tilde S u} du \hs = \hs \frac{2r e^{2Lr\sqrt{2/q\tilde S}}}{q\tilde S}.
\end{eqnarray*}
Using the identity from Lemma \ref{extinction_lem} together with Jensen's inequality gives
\begin{eqnarray*}
\Pb_x(Z^{g_t,L}(\infty)>0) = \Qb_x\left[\frac{Z^{g_t,L}(0)}{Z^{g_t,L}(\infty)}\right] &=& \Qt_x\left[ \Qt_x\left[\left. \frac{1}{Z^{g_t,L}(\infty)} \right| \Gg_\infty \right]\right] \cos(\frac{\pi x}{2L}) \\
&\geq& \Qt_x\left[  \frac{1}{\Qt_x[Z^{g_t,L}(\infty)|\Gg_\infty]} \right] \cos(\frac{\pi x}{2L})\\
&\geq& \frac{q\tilde S}{2r e^{2Lr\sqrt{2/q\tilde S}}}\cos(\frac{\pi x}{2L}).
\end{eqnarray*}
Thus we have shown that if $t\in U\cap V$ then $\Pb_x(Z^{g_t,L}(\infty)>0)$ is large enough for all $x\in[-L+\delta,L-\delta]$, and it now suffices to show that for $t\geq T(p)$,
\[\int_0^t \ind_{U\cap V}(s) ds \geq (p-2q)\frac{\tilde S}{r}t.\]
But if $t\geq T(p)$ then, since $J$ increases at rate at most $r$,
\[(p-q)\tilde S t \leq J_t \leq \int_0^t r\ind_U(s) ds,\]
and (in fact whenever $t\geq T(0)$)
\[2rt \geq \int_0^t f'(s)^2 ds \geq \int_0^t \frac{2r^2}{q\tilde S} \ind_{V^c}(s) ds;\]
hence for any $t\geq T(p)$
\[\int_0^t \ind_{U\cap V}(s) ds \geq \int_0^t \ind_U(s) ds - \int_0^t \ind_{V^c}(s)ds \geq (p-q)\frac{\tilde S}{r}t - q\frac{\tilde S}{r}t = (p-2q)\frac{\tilde S}{r}t\]
as required.
\end{proof}

We now show that if a particle has remained in the tube for a long time, then it is very likely to have spent a long time in $U_\alpha$. The idea is that if $U_\alpha$ stretches to within $\delta$ of the edge of the tube for a proportion of time, then in order to stay out of $U_\alpha$ a particle must spend a long time in a tube of radius $\delta$. We give estimates for the time spent by Brownian motion in such a tube and apply these to our problem via the many-to-one theorem (Theorem \ref{many_to_one}).

\begin{lem}\label{spine_U}
Fix $\beta<1$ and $\gamma>0$. If $\tilde S >0$ then for sufficiently small $\alpha>0$ and large $T$, we have
\[\Pb(\exists u \in\hat N(t) : I_\alpha(u;t)< \beta\frac{\tilde S}{r}t) \leq e^{-\gamma t}.\]
\end{lem}

\begin{proof}
First we show that for any $\delta > 0$ and $k>0$,
\[\Pt(\int_0^t \ind_{\{\xi_s \in (-\delta, \delta)\}} ds > k) \leq 3e^{t/2 - k/4\delta}.\]
Recall that under $\Pt$, the spine's motion is simply a Brownian motion. One may check (by approximating with $C^2$ functions and applying It\^o's formula) that, setting
\[h_\delta(x) = \left\{\begin{array}{ll} |x| & \hbox{if } |x|\geq \delta \\ \frac{\delta}{2} + \frac{x^2}{2\delta} & \hbox{if } |x|<\delta \end{array}\right.\]
we have
\[h_\delta(\xi_t) = \frac{\delta}{2} + \int_0^t h'_\delta (\xi_s) d\xi_s + \frac{1}{2\delta}\int_0^t \ind_{\{\xi_s\in(-\delta,\delta)\}} ds.\]
Also,
\[\Pt[e^{-\int_0^t h'_\delta(\xi_s)d\xi_s}] \leq \Pt[e^{-\int_0^t h'_\delta(\xi_s)d\xi_s - \frac{1}{2}\int_0^t h'_\delta(\xi_s)^2 ds}]e^{t/2} \leq e^{t/2}.\]
Thus
\begin{eqnarray*}
&&\Pt(\int_0^t \ind_{\{\xi_s \in (-\delta,\delta)\}} ds > k)\\
&&= \Pt(h_\delta(\xi_t) - \frac{\delta}{2} - \int_0^t h_\delta '(\xi_s) d\xi_s > \frac{k}{2\delta})\\
&&\leq \Pt(|\xi_t| - \int_0^t h_\delta ' (\xi_s) d\xi_s > \frac{k}{2\delta})\\
&&\leq \Pt(\xi_t > \frac{k}{4\delta}) + \Pt(-\xi_t > \frac{k}{4\delta}) + \Pt(- \int_0^t h_\delta ' (\xi_s) d\xi_s > \frac{k}{4\delta})\\
&&\leq \Pt[e^{\xi_t}]e^{-k/4\delta} + \Pt[e^{-\xi_t}]e^{-k/4\delta} + \Pt[e^{- \int_0^t h_\delta ' (\xi_s) d\xi_s}]e^{-k/4\delta}\\
&&\leq 3e^{t/2 - k/4\delta},
\end{eqnarray*}
establishing our first claim. 
Now, for any $\delta>0$, by Lemma \ref{tube_filling} we may choose $\alpha>0$ and $T$ such that
\[\int_0^t \ind_{\{(s,x)\in U_\alpha \hsl \forall x\in[-L+\delta,L-\delta]\}} ds \geq (\frac{1+\beta}{2})\frac{\tilde S}{r}t \hs \forall t\geq T.\]
Then if the spine particle is to have spent less than $\beta\frac{\tilde S}{r}t$ time in $U_\alpha$ (yet remained within the tube of width $L$) then it must have spent at least $(\frac{1-\beta}{2})\frac{\tilde S}{r}t$ within $\delta$ of the edge of the tube (provided that $t$ is large enough). That is, for $t\geq T$,
\begin{eqnarray*}
&&\Pt(\xi_t \in \hat N(t), I_\alpha(\xi_t;t)<\beta\frac{\tilde S}{r}t)\\
&&\leq \Pt(\int_0^t \ind_{\{\xi_s \in (f(s)-L, f(s)-L+\delta)\cup(f(s)+L-\delta, f(s)+L)\}}ds > (\frac{1-\beta}{2})\frac{\tilde S}{r}t).
\end{eqnarray*}
In fact, using the fact that if $\xi_t \in \hat N(t)$ then we may apply the Girsanov part of our usual measure change and our usual estimate on it,
\begin{eqnarray*}
&&\Pt(\xi_t \in \hat N(t), I_\alpha(\xi_t;t)<\beta\frac{\tilde S}{r}t)\\
&&\leq \Pt\left[\frac{\ind_{\{\xi_t \in \hat N(t)\}}}{e^{\int_0^t f'(s)d\xi_s - \frac{1}{2}\int_0^t f'(s)^2 ds}}\ind_{\{\int_0^t \ind_{\{\xi_s \in (-L, -L+\delta)\cup(L-\delta,L)\}}ds > (\frac{1-\beta}{2})\frac{\tilde S}{r}t\}}\right]\\
&&\leq e^{2L\int_0^t |f''(s)| ds + 2L|f'(0)|}\Pt(\int_0^t \ind_{\{\xi_s \in (-L, -L+\delta)\cup(L-\delta,L)\}}ds > (\frac{1-\beta}{2})\frac{\tilde S}{r}t).
\end{eqnarray*}
By the reflection and Markov properties of Brownian motion, we have
\begin{multline*}
\hspace{10mm}\Pt(\int_0^t \ind_{\{\xi_s \in (-L, -L+\delta)\cup(L-\delta,L)\}}ds > (\frac{1-\beta}{2})\frac{\tilde S}{r}t)\\
\leq 2\Pt(\int_0^t \ind_{\{\xi_s \in (-\delta,\delta)\}} ds > (\frac{1-\beta}{4})\frac{\tilde S}{r}t).\hspace{10mm}
\end{multline*}
Putting all of this together and using the estimate given in the first part of the proof, we get
\[\Pt(\xi_t \in \hat N(t), I_\alpha(\xi_t;t)<\beta\frac{\tilde S}{r}t) \leq 2e^{2L\int_0^t |f''(s)| ds + 2L|f'(0)|}.3e^{\frac{t}{2}-(\frac{1-\beta}{4})\frac{\tilde S}{r}t/4\delta}.\]
Finally, taking $\delta = \frac{(1-\beta)\tilde S}{16r(r+\gamma+1)}$ and using the fact that for $t\geq T(0)$ we have $e^{2L\int_0^t |f''(s)| ds + 2L|f'(0)|}\leq e^{rt}$, we get
\[\Pt(\xi_t \in \hat N(t), I_\alpha(\xi_t;t)<\beta\frac{\tilde S}{r}t) \leq e^{-\gamma t} \hs \forall t\geq T\vee T(0)\vee 2\log 6.\qedhere\]
\end{proof}

\begin{prop}\label{invspineprop}
Recall that $\Upsilon$ is the extinction time for the process. If $\tilde S > 0$ then
\[\Pb(\Upsilon = \infty) = \Pb(Z(\infty)>0).\]
\end{prop}

\begin{proof}
We note that $\{Z(\infty)>0\} \subseteq \{\Upsilon = \infty\}$, so it suffices to show that for any $\varepsilon >0$,
\[\Pb(\Upsilon = \infty, \hs Z(\infty)=0) < \varepsilon.\]
To this end, fix $\varepsilon > 0$ and choose $\alpha$ small enough and $T_0$ large enough that
\[\Pb(\exists u \in\hat N(t) : I_\alpha(u;t)< \frac{\tilde S}{2r}t) < \varepsilon/3 \hs \forall t\geq T_0\]
(this is possible by Lemma \ref{spine_U}). Choose an integer $m$ large enough that $(1-\alpha)^m < \varepsilon/3$. Finally, choose $T\geq T_0$ large enough that
\[\sum_{j=0}^{m-1} \frac{e^{-\tilde S T / 2}(\tilde S T / 2)^j}{j!} < \varepsilon/3.\]
Then
\begin{eqnarray*}
\Pb(\Upsilon = \infty, \hs Z(\infty)=0) &\leq& \Pb(\exists u \in \hat N(T), \hs Z(\infty)=0)\\
&<& \Pb(\exists u \in \hat N(T), \hsl I_\alpha(u;T) \geq \frac{\tilde S}{2r}T, \hsl Z(\infty)=0) + \varepsilon/3.
\end{eqnarray*}
Now, if a particle $u$ has spent at least $\frac{\tilde S}{2r}T$ time in $U_\alpha$ then (by the choice of $T$, since the births along $u$ form a Poisson process of rate $r$) it has probability at least $(1-\varepsilon/3)$ of having at least $m$ births whilst in $U_\alpha$. Each of these particles born within $U_\alpha$ launches an independent population from a point $(t,x)\in U_\alpha$, so that
\[Z(\infty)\geq \sum_{v<u} e^{-r S_v} Z_v(\infty) \ind_{\{(S_v, X_u(S_v))\in U_\alpha\}}\]
where each $Z_v$ is a non-negative martingale on the interval $[S_v,\infty)$ with law equal to that of $Z^{g_t}$ started from $x$, and hence satisfying \mbox{$\Pb(Z_v(\infty)>0) \geq \alpha$}. Thus
\begin{eqnarray*}
&&\Pb(\Upsilon = \infty, \hs Z(\infty)=0)\\
&&\leq \Pb(\exists u \in \hat N(T), \hs I_\alpha(u;T) \geq \frac{\tilde S}{2r}T, \hs Z(\infty)=0) + \varepsilon/3\\
&&\leq \Pb\left(\exists u \in \hat N(T), \hs \left\{\begin{array}{c} u \hbox{ has had at least}\\ m \hbox{ births within } U_\alpha\end{array}\right\}, \hs Z(\infty)=0\right) + 2\varepsilon/3\\
&&\leq (1-\alpha)^m + 2\varepsilon/3 \hs < \hs \varepsilon
\end{eqnarray*}
which completes the proof.
\end{proof}

We draw our results together as follows.

\begin{mainprf}
All that remains is to combine Proposition \ref{uiprop} with Propositions \ref{liminf_prop} and \ref{limsup_prop} to gain the desired growth bounds; Proposition \ref{invspineprop} guarantees that we are working on the correct set.\qed
\end{mainprf}

\section{Extending the class of functions}\label{extensions}
As we mentioned earlier, the usual conditions on the function $f$ (specifically the smoothness requirements) in Theorem \ref{mainthm} may be weakened by approximating uniformly and checking that the relevant quantities converge as desired. To see this, suppose that we have a function $f$ which does not satisfy the usual conditions, but such that we have a sequence of functions $f_n:[0,\infty)\to\Rb$, each satisfying the usual conditions, converging uniformly to $f$. Let
\[\bar{S}:= \limsup_{n\to\infty}\limsup_{t\to\infty}\frac{1}{t}\int_0^t f_n'(s)^2 ds\]
and
\[\underline{S}:= \liminf_{n\to\infty}\liminf_{t\to\infty}\frac{1}{t}\int_0^t f_n'(s)^2 ds.\]

\begin{cor}\label{maincor}
If $r < \frac{\bar S}{2} + \frac{\pi^2}{8L^2}$, then $\Upsilon<\infty$ almost surely.
On the other hand, if $r > \frac{\bar S}{2} + \frac{\pi^2}{8L^2}$, then $\Pb(\Upsilon=\infty)>0$ and almost surely on non-extinction we have
\[\limsup_{t\to\infty} \frac{1}{t}\log|\hat N^{f,L}(t)| = r - \frac{\pi^2}{8L^2} - \frac{1}{2}\hspace{0.5mm}\underline S\]
and
\[\liminf_{t\to\infty} \frac{1}{t}\log|\hat N^{f,L}(t)| = r - \frac{\pi^2}{8L^2} - \frac{1}{2}\hspace{0.5mm}\bar S.\]
\end{cor}

\begin{proof}
This follows easily from Theorem \ref{mainthm} by letting
\[\bar L_n = L + ||f-f_n||_\infty \hs \hbox{ and } \hs \underline L_n = L - ||f-f_n||_\infty\]
and noting that for each $n\geq1$ and $t\geq0$,
\[(f(t) - L, f(t) + L) \subseteq (f_n(t) - \bar L_n, f_n(t) + \bar L_n)\]
and
\[(f(t) - L, f(t) + L) \supseteq (f_n(t) - \underline L_n, f_n(t) + \underline L_n).\qedhere\]
\end{proof}

\noindent
Even with this extension to our theorem, however, there are some functions that still escape our net: for example, $f(t)=\sin t$ is a particularly nice function that one might wish our theorem to cover. In fact, the following example demonstrates that the usual growth rate cannot hold in all cases: 

\begin{ex}\label{badex1}
Let
\[f_\delta(t):=\delta\sin(t/\delta);\]
then as $\delta\to0$, $f_\delta$ converges uniformly to the zero function, $f(t)\equiv 0$.
By Theorem \ref{mainthm} we know that on survival,
\[\lim_{t\to\infty} \frac{1}{t}\log|\hat N^{f,L}(t)| = r - \frac{\pi^2}{8L^2}.\]
However, if the result of Theorem \ref{mainthm} held for each $f_\delta$ then by the same argument as in Corollary \ref{maincor} we would have (on survival)
\[\lim_{t\to\infty} \frac{1}{t}\log|\hat N^{f,L}(t)| = r - \frac{\pi^2}{8L^2} - \frac{1}{4}.\]
Of course, $f_\delta$ does not satisfy usual condition (3) and hence this contradiction does not appear -- it simply serves to highlight the fact that our result cannot hold without some condition on the second derivative.
\end{ex}

\begin{ex}\label{badex2}
Another interesting example is given by letting
\[g_\delta(t):=\sin(t/\delta).\]
Again $g_\delta$ does not satisfy usual condition (3) and we cannot apply Theorem \ref{mainthm}. However, as $\delta\to0$, the frequency of the oscillations increases while the amplitude stays constant, and we expect that the number of particles staying within $L$ of $g_\delta$ should be approximately equal to the number staying within $L-1$ of the constant zero function: that is, we expect for small $\delta$
\[\lim_{t\to\infty} \frac{1}{t}\log|\hat N^{g_\delta,L}(t)| \approx r - \frac{\pi^2}{8(L-1)^2}.\]
\begin{figure}[h!]
  \centering
      \includegraphics[width=9.657cm]{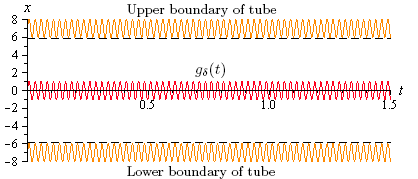}
  \caption{$g_\delta(t)=\sin(t/\delta)$ for small $\delta>0$, $L=7$}
\end{figure}
\end{ex}
Motivated by examples \ref{badex1} and \ref{badex2}, we hope to consider extensions of our result in the future. However it is not clear whether almost sure values for the limsup and liminf even exist in all cases.

\bibliographystyle{plain}
\bibliography{my_bibliography}

\end{document}